\documentclass[a4paper,twoside,english,10pt]{amsart}

\usepackage[cm]{fullpage}
\usepackage{amsmath,amsfonts,amssymb}
\usepackage[all]{xy}
\usepackage[english]{babel}
\usepackage[utf8x]{inputenc}
\usepackage{url}
\usepackage{ifthen}

 \newtheorem{thm}{Theorem}[section]
 \newtheorem{cor}[thm]{Corollary}

\theoremstyle{definition}
 \newtheorem{rem}[thm]{Remark}

\newcommand{\N}{\mathbb{N}}

\newcommand{\R}{\mathbb{R}}
\newcommand{\C}{\mathbb{C}}

\newcommand{\st}{\;:\;}

\newcommand{\sspace}{\cdot}

\newcommand{\End}[1]{\mathrm{End}\left(#1\right)}

\newcommand{\paragrafo}[2]{\smallskip \noindent \texttt{Step {#1}} -- {\itshape #2}.\ }

\newcommand{\kth}[1]{\ifthenelse{\equal{#1}{1}}{$#1^\text{st}$}{\ifthenelse{\equal{#1}{2}}{$#1^\text{nd}$}{\ifthenelse{\equal{#1}{3}}{$#1^\text{rd}$}{$#1^\text{th}$}}}}

\newcommand{\correnti}{\mathcal{D}}

\DeclareMathOperator{\imm}{im}
\DeclareMathOperator{\de}{d}
\DeclareMathOperator{\id}{id}
\DeclareMathOperator{\GL}{GL}
\DeclareMathOperator{\Sing}{Sing}

\DeclareMathOperator{\pr}{pr}

\DeclareMathOperator{\ord}{ord}

\newcommand{\del}{\partial}
\newcommand{\delbar}{\overline{\del}}

\allowdisplaybreaks[1]

\title{Cohomologies of certain orbifolds}
\author{Daniele Angella}
\address{Dipartimento di Matematica\\
Universit\`{a} di Pisa \\
Largo Bruno Pontecorvo 5, 56127\\ 
Pisa, Italy}
\email{angella@mail.dm.unipi.it}

\keywords{Bott-Chern cohomology, orbifolds, $\partial\overline{\partial}$-Lemma}
\thanks{This work was supported by GNSAGA of INdAM}
\subjclass[2010]{55N32, 32Q99, 32C15}

\begin{document}

\begin{abstract}
 We study the Bott-Chern cohomology of complex orbifolds obtained as quotient of a compact complex manifold by a finite group of biholomorphisms.
\end{abstract}

\maketitle

\section*{Introduction}

In order to investigate cohomological aspects of compact complex non-K\"ahler manifolds, and in particular with the aim to get results allowing to construct new examples of non-K\"ahler manifolds, we study the cohomology of complex orbifolds.

\medskip

Namely, an \emph{orbifold} (or \emph{V-manifold}, as introduced by I. Satake,  \cite{satake}) is a singular complex space whose singularities are locally isomorphic to quotient singularities $\left. \C^n\right\slash G$, for finite subgroups $G\subset \GL(n;\C)$, where $n$ is the complex dimension: in other words, local geometry of orbifolds reduces to local $G$-invariant geometry. A special case is provided by orbifolds of global-quotient type, namely, by orbifolds $\tilde X \;=\; \left. X \right\slash G$, where $X$ is a complex manifold and $G$ is a finite group of biholomorphisms of $X$; such orbifolds have been studied, among others, by D.~D. Joyce in constructing examples of compact manifolds with special holonomy, see \cite{joyce-jdg-1-2-g2, joyce-invent, joyce-jdg-spin7, joyce-ext}. As proven by I. Satake, and W.~L. Baily, from the cohomological point of view, one can adapt both the sheaf-theoretic and the analytic tools for the study of the de Rham and Dolbeault cohomology of complex orbifolds, \cite{
satake, baily, baily-2}.

In particular, an useful tool in studying the cohomological properties of non-K\"ahler manifolds is provided by the \emph{Bott-Chern cohomology}, that is, the bi-graded algebra
$$ H^{\bullet,\bullet}_{BC}(X) \;:=\; \frac{\ker \del \cap \ker \delbar}{\imm \del\delbar} \;.$$
While for compact K\"ahler manifolds $X$ one has that the Bott-Chern cohomology is naturally isomorphic to the Dolbeault cohomology, \cite[Lemma 5.15, Remark 5.16, 5.21, Lemma 5.11]{deligne-griffiths-morgan-sullivan}, in general, for compact non-K\"ahler manifolds $X$, the natural maps $H^{\bullet,\bullet}_{BC}(X)\to H^{\bullet,\bullet}_{\delbar}(X)$ and $H^{\bullet,\bullet}_{BC}(X)\to H^{\bullet}_{dR}(X;\C)$ induced by the identity are neither injective nor surjective. One says that a compact complex manifold \emph{satisfies the $\del\delbar$-Lemma} if every $\del$-closed $\delbar$-closed $\de$-exact form is $\del\delbar$-exact, that is, if the natural map $H^{\bullet,\bullet}_{BC}(X)\to H^{\bullet}_{dR}(X;\C)$ is injective; compact K\"ahler manifolds provide the main examples of complex manifolds satisfying the $\del\delbar$-Lemma, \cite[Lemma 5.11]{deligne-griffiths-morgan-sullivan}, other than motivations for their study.

\medskip

In this note, we study the \emph{Bott-Chern cohomology} of compact complex orbifolds $\tilde X = \left.X \right\slash G$ of global-quotient type, (where $X$ is a compact complex manifold and $G$ is a finite group of biholomorphisms of $X$,) that is, the bi-graded $\C$-algebra
$$ H^{\bullet,\bullet}_{BC}\left(\tilde X\right) \;:=\; \frac{\ker \del\cap \ker\delbar}{\imm \del\delbar} $$
where $\del\colon \wedge^{\bullet,\bullet}\tilde X\to \wedge^{\bullet+1,\bullet}\tilde X$ and $\delbar\colon \wedge^{\bullet,\bullet}\tilde X\to \wedge^{\bullet,\bullet+1}\tilde X$, and $\wedge^{\bullet,\bullet}\tilde X$ is the bi-graded $\C$-vector space of \emph{differential forms} on $\tilde X$, that is, of $G$-invariant differential forms on $X$. We prove the following result, see Theorem \ref{thm:bc}.

\smallskip
\noindent{\bfseries Theorem.} {\itshape
 Let $\tilde X=\left.X \right\slash G$ be a compact complex orbifold of complex dimension $n$, where $X$ is a compact complex manifold and $G$ is a finite group of biholomorphisms of $X$.
 For any $p,q\in\N$, there is a canonical isomorphism
 $$
  H^{p,q}_{BC}\left(\tilde X\right) \;\simeq\; \frac{\ker\left(\del \colon \correnti^{p,q}\tilde X \to \correnti^{p+1,q}\tilde X\right) \cap \ker \left(\delbar\colon \correnti^{p,q}\tilde X \to \correnti^{p,q+1}\tilde X\right)}{\imm\left(\del\delbar \colon \correnti^{p-1,q-1}\tilde X \to \correnti^{p,q}\tilde X\right)} \;,
 $$
 where $\correnti^{p,q}\tilde X$ denotes the space of currents of bi-degree $(p,q)$ on $\tilde X$, that is, the space of $G$-invariant currents of bi-degree $(p,q)$ on $X$.

 Furthermore, given a Hermitian metric on $\tilde X$ (that is, a $G$-invariant Hermitian metric on $X$), there are canonical isomorphisms
 $$ H^{\bullet,\bullet}_{BC}\left(\tilde X\right) \;\simeq\; \ker \tilde\Delta_{BC} \qquad \text{ and } \qquad H^{\bullet,\bullet}_{A}\left(\tilde X\right) \;\simeq\; \ker\tilde\Delta_A \;, $$
 where $\tilde\Delta_{BC}$ and $\tilde\Delta_{A}$ are the \kth{4} order self-adjoint elliptic differential operators
 $$ \tilde \Delta_{BC} \;:=\; \left(\del\delbar\right)\left(\del\delbar\right)^*+\left(\del\delbar\right)^*\left(\del\delbar\right)+\left(\delbar^*\del\right)\left(\delbar^*\del\right)^*+\left(\delbar^*\del\right)^*\left(\delbar^*\del\right)+\delbar^*\delbar+\del^*\del \;\in\; \End{\wedge^{\bullet, \bullet}\tilde X} $$
 and
 $$ \tilde\Delta_{A} \;:=\; \del\del^*+\delbar\delbar^*+\left(\del\delbar\right)^*\left(\del\delbar\right)+\left(\del\delbar\right)\left(\del\delbar\right)^*+\left(\delbar\del^*\right)^*\left(\delbar\del^*\right)+\left(\delbar\del^*\right)\left(\delbar\del^*\right)^* \;\in\; \End{\wedge^{\bullet, \bullet}\tilde X} \;. $$

 In particular, the Hodge-$*$-operator induces an isomorphism
 $$ H^{\bullet_1,\bullet_2}_{BC}\left(\tilde X\right) \;\simeq\; H^{n-\bullet_2,n-\bullet_1}_{A}\left(\tilde X\right) \;.$$
}
\smallskip

As regards the $\del\delbar$-Lemma for complex orbifolds, by adapting a result by R.~O. Wells in \cite{wells}, we get the following result, see Corollary \ref{cor:deldelbar-lemma}.

\smallskip
\noindent{\bfseries Corollary.} {\itshape
 Let $\tilde Y$ and $\tilde X$ be compact complex orbifolds of the same complex dimension, and let $\epsilon\colon \tilde Y\to \tilde X$ be a proper surjective morphism of complex orbifolds. If $\tilde Y$ satisfies the $\del\delbar$-Lemma, then also $\tilde X$ satisfies the $\del\delbar$-Lemma.
}
\smallskip

\medskip

\noindent{\sl Acknowledgments.} The author would like to warmly thank Adriano Tomassini, both for his constant support and encouragement, and for many useful discussions and suggestions. Thanks are also due to Marco Abate for several remarks that improved the presentation of this note.

\section{Preliminaries on orbifolds}\label{sec:orbifolds}

The notion of orbifold has been introduced by I. Satake in \cite{satake}, with the name of \emph{V-manifold}, and has been studied, among many others, by W.~L. Baily, \cite{baily, baily-2}.

In this section, we start by recalling the main definitions and some classical results concerning complex orbifolds and their cohomology, referring to \cite{joyce-red, joyce-ext, satake, baily, baily-2}.

\medskip

A \emph{complex orbifold of complex dimension $n$} is a singular complex space whose singularities are locally isomorphic to quotient singularities $\left. \C^n\right\slash G$, for finite subgroups $G\subset \GL(n;\C)$, \cite[Definition 2]{satake}.

By definition, an object (e.g., a \emph{differential form}, a \emph{Riemannian metric}, a \emph{Hermitian metric}) \emph{on a complex orbifold $\tilde X$} is defined locally at $x\in\tilde X$ as a $G_x$-invariant object on $\C^n$, where $G_x\subseteq\GL(n;\C)$ is such that $\tilde X$ is locally isomorphic to $\left. \C^n \right\slash G_x$ at $x$.

Given $\tilde X$ and $\tilde Y$ complex orbifolds, a \emph{morphism} $f\colon \tilde Y \to \tilde X$ \emph{of complex orbifolds} is a morphism of complex spaces given, locally at $y\in \tilde Y$, by a map $\left.\C^m\right\slash H_y \to \left.\C^n\right\slash G_{f(y)}$, where $\tilde Y$ is locally isomorphic to $\left.\C^m\right\slash H_y$ at $y$ and $\tilde X$ is locally isomorphic to $\left.\C^n\right\slash G_{f(y)}$ at $f(y)$.

\medskip

In particular, one gets a differential complex $\left(\wedge^\bullet \tilde X, \, \de\right)$, and a double complex $\left(\wedge^{\bullet,\bullet}\tilde X,\, \del,\, \delbar\right)$. Define the de Rham, Dolbeault, Bott-Chern, and Aeppli cohomology groups of $\tilde X$ respectively as
\begin{eqnarray*}
   H^\bullet_{dR}\left(\tilde X;\C\right) \;:=\; \frac{\ker \de}{\imm\de} \;, &\qquad&
   H^{\bullet,\bullet}_{\delbar}\left(\tilde X\right) \;:=\; \frac{\ker \delbar}{\imm \delbar} \;, \\[5pt]
   H^{\bullet,\bullet}_{BC}\left(\tilde X\right) \;:=\; \frac{\ker \del \cap \ker \delbar}{\imm \del\delbar} \;, &\qquad&
   H^{\bullet,\bullet}_{A}\left(\tilde X\right) \;:=\; \frac{\ker \del\delbar}{\imm \del + \imm \delbar} \;.
\end{eqnarray*}
The structure of double complex of $\left(\wedge^{\bullet, \bullet}\tilde X,\, \del,\, \delbar\right)$ induces naturally a spectral sequence $\left\{\left(E_r^{\bullet,\bullet},\, \de_r\right)\right\}_{r\in\N}$, called \emph{Hodge and Fr\"olicher spectral sequence of $\tilde X$}, such that $E_1^{\bullet,\bullet}\simeq H^{\bullet,\bullet}_{\delbar}\left(\tilde X\right)$ (see, e.g., \cite[\S2.4]{mccleary}). Hence, one has the \emph{Fr\"olicher inequality}, see \cite[Theorem 2]{frolicher},
$$ \sum_{p+q=k} \dim_\C H^{p,q}_{\delbar}\left(\tilde X \right) \;\geq\; \dim_\C H^k_{dR}\left(\tilde X;\C\right) \;,$$
for any $k\in\N$.

\medskip

Given a Riemannian metric on a complex orbifold $\tilde X$ of complex dimension $n$, one can consider the $\R$-linear Hodge-$*$-operator $*_g\colon \wedge^\bullet \tilde X \to \wedge^{2n-\bullet}\tilde X$, and hence the \kth{2} order self-adjoint elliptic differential operator $\Delta:=\left[\de,\, \de^*\right]:=\de\,\de^*+\de^*\,\de \in \End{\wedge^\bullet \tilde X}$.

Analogously, given a Hermitian metric on a complex orbifold $\tilde X$ of complex dimension $n$, one can consider the $\C$-linear Hodge-$*$-operator $*_g\colon \wedge^{\bullet_1, \bullet_2} \tilde X \to \wedge^{n-\bullet_2, n-\bullet_1}\tilde X$, and hence the \kth{2} order self-adjoint elliptic differential operator $\overline\square:=\left[\delbar,\, \delbar^*\right]:=\delbar\,\delbar^*+\delbar^*\,\delbar \in \End{\wedge^{\bullet,\bullet}\tilde X}$. Furthermore, in \cite[Proposition 5]{kodaira-spencer-3}, and \cite[\S2]{schweitzer}, the following \kth{4} order self-adjoint elliptic differential operators are defined:
$$ \tilde \Delta_{BC} \;:=\; \left(\del\delbar\right)\left(\del\delbar\right)^*+\left(\del\delbar\right)^*\left(\del\delbar\right)+\left(\delbar^*\del\right)\left(\delbar^*\del\right)^*+\left(\delbar^*\del\right)^*\left(\delbar^*\del\right)+\delbar^*\delbar+\del^*\del \;\in\; \End{\wedge^{\bullet, \bullet}\tilde X} $$
and
$$ \tilde\Delta_{A} \;:=\; \del\del^*+\delbar\delbar^*+\left(\del\delbar\right)^*\left(\del\delbar\right)+\left(\del\delbar\right)\left(\del\delbar\right)^*+\left(\delbar\del^*\right)^*\left(\delbar\del^*\right)+\left(\delbar\del^*\right)\left(\delbar\del^*\right)^* \;\in\; \End{\wedge^{\bullet, \bullet}\tilde X} \;. $$

As a matter of notation, given a compact complex orbifold $\tilde X$ of complex dimension $n$, denote the constant sheaf with coefficients in $\R$ over $\tilde X$ by $\underline{\R}_{\tilde X}$, the sheaf of germs of smooth functions over $\tilde X$ by $\mathcal{C}^{\infty}_{\tilde X}$, the sheaf of germs of $(p,q)$-forms (for $p,q\in\N$) over $\tilde X$ by $\mathcal{A}^{p,q}_{\tilde X}$, the sheaf of germs of $k$-forms (for $k\in\N$) over $\tilde X$ by $\mathcal{A}^{k}_{\tilde X}$, the sheaf of germs of bidimension-$(p,q)$-currents (for $p,q\in\N$) over $\tilde X$ by $\mathcal{D}_{\tilde X\, p,q}:=:\mathcal{D}^{n-p,n-q}_{\tilde X}$, the sheaf of germs of dimension-$k$-currents (for $k\in\N$) over $\tilde X$ by $\mathcal{D}_{\tilde X\, k}:=:\mathcal{D}^{2n-k}_{\tilde X}$, and the sheaf of holomorphic $p$-forms (for $p\in\N$) over $\tilde X$ by $\Omega^p_{\tilde X}$.

\medskip

The following result, concerning the de Rham cohomology of a compact complex orbifold, has been proven by I. Satake, \cite{satake}, and by W.~L. Baily, \cite{baily}.

\begin{thm}[{\cite[Theorem 1]{satake}, \cite[Theorem H]{baily}}]
 Let $\tilde X$ be a compact complex orbifold of complex dimension $n$.
 There is a canonical isomorphism
 $$ H^\bullet_{dR}\left(\tilde X;\R\right) \;\simeq\; \check H^\bullet\left(\tilde X; \underline{\R}_{\tilde X}\right) \;.$$

 Furthermore, given a Riemannian metric on $\tilde X$, there is a canonical isomorphism
 $$ H^\bullet_{dR}\left(\tilde X;\R\right) \;\simeq\; \ker \Delta \;.$$

 In particular, the Hodge-$*$-operator induces an isomorphism
 $$ H^\bullet_{dR}\left(\tilde X;\R\right) \;\simeq\; H^{2n-\bullet}_{dR}\left(\tilde X;\R\right) \;.$$
\end{thm}

The isomorphism $H^\bullet_{dR}\left(\tilde X;\R\right) \simeq \ker \Delta$ can be seen as a consequence of a more general decomposition theorem on compact orbifolds, \cite[Theorem D]{baily}, which holds for \kth{2} order self-adjoint elliptic differential operators. In particular, as regards the Dolbeault cohomology, the following result by W.~L. Baily, \cite{baily-2, baily}, holds.

\begin{thm}[{\cite[page 807]{baily-2}, \cite[Theorem K]{baily}}]
 Let $\tilde X$ be a compact complex orbifold of complex dimension $n$.
 There is a canonical isomorphism
 $$ H^{\bullet_1,\bullet_2}_{\delbar}\left(\tilde X\right) \;\simeq\; \check H^{\bullet_2} \left(\tilde X; \Omega^{\bullet_1}_{\tilde X}\right) \;.$$

 Furthermore, given a Hermitian metric on $X$, there is a canonical isomorphism
 $$ H^{\bullet,\bullet}_{\delbar}\left(\tilde X\right) \;\simeq\; \ker \overline\square \;.$$

 In particular, the Hodge-$*$-operator induces an isomorphism
 $$ H^{\bullet_1,\bullet_2}_{\delbar}\left(\tilde X\right) \;\simeq\; H^{n-\bullet_1,n-\bullet_2}_{\delbar}\left(\tilde X\right) \;.$$
\end{thm}

\section{Bott-Chern cohomology of complex orbifolds of global-quotient type}
Compact complex orbifolds of the type $\tilde X = \left. X \right\slash G$, where $X$ is a compact complex manifold and $G$ is a finite group of biholomorphisms of $X$, constitute one of the simplest examples of singular manifolds: more precisely, in this section, we study the Bott-Chern cohomology for such orbifolds, proving that it can be defined using either currents or forms, or also by computing the $G$-invariant $\tilde \Delta_{BC}$-harmonic forms on $X$, Theorem \ref{thm:bc}.

\medskip

Consider
$$ \tilde X \;=\; \left. X \right\slash G \;,$$
where $X$ is a compact complex manifold and $G$ is a finite group of biholomorphisms of $X$: by the Bochner linearization theorem, \cite[Theorem 1]{bochner}, see also \cite[Theorem 1.7.2]{raissy-master-thesis}, $\tilde X$ turns out to be an orbifold as in I. Satake's definition.

Such orbifolds of global-quotient type have been considered and studied by D.~D. Joyce in constructing examples of compact $7$-dimensional manifolds with holonomy $G_2$, \cite{joyce-jdg-1-2-g2} and \cite[Chapters 11-12]{joyce-ext}, and examples of compact $8$-dimensional manifolds with holonomy ${\rm Spin}(7)$, \cite{joyce-invent, joyce-jdg-spin7} and \cite[Chapters 13-14]{joyce-ext}. See also \cite{fernandez-munoz, cavalcanti-fernandez-munoz} for the use of orbifolds of global-quotient type to construct a compact $8$-dimensional simply-connected non-formal symplectic manifold (which do not satisfy, respectively satisfy, the Hard Lefschetz condition), answering to a question by I.~K. Babenko and I.~A. Ta\u{\i}manov, \cite[Problem]{babenko-taimanov}.

Since $G$ is a finite group of biholomorphisms, the singular set of $\tilde X$ is
$$ \Sing\left(\tilde X\right) \;=\; \left\{x\,G\in\left.X\right\slash G 
\st x\in X \text{ and }g\cdot x=x \text{ for some }g\in G\setminus\{\id_X\}\right\} \;.$$

\medskip

We provide the following result, concerning Bott-Chern and Aeppli cohomologies of compact complex orbifolds of global-quotient type.

\begin{thm}\label{thm:bc}
 Let $\tilde X=\left.X \right\slash G$ be a compact complex orbifold of complex dimension $n$, where $X$ is a compact complex manifold and $G$ is a finite group of biholomorphisms of $X$.
 For any $p,q\in\N$, there is a canonical isomorphism
 \begin{equation}\label{eq:bc-orbifolds}
  H^{p,q}_{BC}\left(\tilde X\right) \;\simeq\; \frac{\ker\left(\del \colon \correnti^{p,q}\tilde X \to \correnti^{p+1,q}\tilde X\right) \cap \ker \left(\delbar\colon \correnti^{p,q}\tilde X \to \correnti^{p,q+1}\tilde X\right)}{\imm\left(\del\delbar \colon \correnti^{p-1,q-1}\tilde X \to \correnti^{p,q}\tilde X\right)} \;.
 \end{equation}

 Furthermore, given a Hermitian metric on $\tilde X$, there are canonical isomorphisms
 $$ H^{\bullet,\bullet}_{BC}\left(\tilde X\right) \;\simeq\; \ker \tilde\Delta_{BC} \qquad \text{ and } \qquad H^{\bullet,\bullet}_{A}\left(\tilde X\right) \;\simeq\; \ker\tilde\Delta_A \;. $$

 In particular, the Hodge-$*$-operator induces an isomorphism
 $$ H^{\bullet_1,\bullet_2}_{BC}\left(\tilde X\right) \;\simeq\; H^{n-\bullet_2,n-\bullet_1}_{A}\left(\tilde X\right) \;.$$
\end{thm}

\begin{proof}
We use the same argument as in the proof of \cite[Theorem 3.7]{angella-1} to show that, since the de Rham cohomology and the Dolbeault cohomology of $\tilde X$ can be computed using either differential forms or currents, the same holds true for the Bott-Chern and the Aeppli cohomologies.

Indeed, note that, for any $p,q\in\N$, one has the exact sequence
\begin{eqnarray*}
\lefteqn{0 \to \frac{\imm \left(\de \colon \left(\correnti^{p+q-1}\tilde X\otimes_\R\C\right)\to \left(\correnti^{p+q}\tilde X\otimes_\R\C\right)\right)\cap \correnti^{p,q}\tilde X}{\imm \left(\del\delbar\colon \correnti^{p-1,q-1}\tilde X \to \correnti^{p,q}\tilde X\right)} } \\[5pt]
 && \to \frac{\ker \left(\de\colon \correnti^{p,q}\tilde X\to \correnti^{p+1,q+1}\tilde X\right)}{\imm \left(\del\delbar\colon \correnti^{p-1,q-1}\tilde X\to \correnti^{p,q}\tilde X\right)}\to \frac{\ker \left(\de \colon \left(\correnti^{p+q}\tilde X\otimes_\R\C\right)\to \left(\correnti^{p+q+1}\tilde X \otimes_\R\C\right)\right)}{\imm \left(\de\colon \left(\correnti^{p+q-1}\tilde X \otimes_\R\C\right) \to \left(\correnti^{p+q}\tilde X \otimes_\R \C\right)\right)} \;,
\end{eqnarray*}
where the maps are induced by the identity.
By \cite[Theorem 1]{satake}, one has
$$ \frac{\ker \left(\de \colon \left(\correnti^{p+q}\tilde X\otimes_\R\C\right)\to \left(\correnti^{p+q+1}\tilde X \otimes_\R\C\right)\right)}{\imm \left(\de\colon \left(\correnti^{p+q-1}\tilde X \otimes_\R\C\right) \to \left(\correnti^{p+q}\tilde X \otimes_\R \C\right)\right)} \;\simeq\; \frac{\ker \left(\de \colon \left(\wedge^{p+q}\tilde X\otimes_\R\C\right)\to \left(\wedge^{p+q+1}\tilde X \otimes_\R\C\right)\right)}{\imm \left(\de\colon \left(\wedge^{p+q-1}\tilde X \otimes_\R\C\right) \to \left(\wedge^{p+q}\tilde X \otimes_\R \C\right)\right)} \;,$$
therefore it suffices to prove that the space
$$ \frac{\imm \left(\de \colon \left(\correnti^{p+q-1}\tilde X\otimes_\R\C\right)\to \left(\correnti^{p+q}\tilde X\otimes_\R\C\right)\right)\cap \correnti^{p,q}\tilde X}{\imm \left(\del\delbar\colon \correnti^{p-1,q-1}\tilde X \to \correnti^{p,q}\tilde X\right)}  $$
can be computed using just differential forms on $\tilde X$.

Firstly, we note that, since, by \cite[page 807]{baily-2},
$$ \frac{\ker \left(\delbar \colon \correnti^{\bullet,\bullet}\tilde X \to \correnti^{\bullet,\bullet+1}\tilde X\right)}{\imm \left(\delbar\colon \correnti^{\bullet,\bullet-1}\tilde X \to \correnti^{\bullet,\bullet}\tilde X \right)} \;\simeq\; \frac{\ker \left(\delbar \colon \wedge^{\bullet,\bullet}\tilde X \to \wedge^{\bullet,\bullet+1}\tilde X\right)}{\imm \left(\delbar\colon \wedge^{\bullet,\bullet-1}\tilde X \to \wedge^{\bullet,\bullet}\tilde X \right)} \;,$$
one has that, if $\psi\in\wedge^{r,s}\tilde X$ is a $\delbar$-closed differential form, then every solution $\phi\in\correnti^{r,s-1}$ of $\delbar\phi=\psi$ is a differential form up to $\delbar$-exact terms.
Indeed, since $[\psi]=0$ in $\frac{\ker\delbar\cap \correnti^{r,s}\tilde X}{\imm\delbar}$ and hence in $\frac{\ker\delbar\cap \wedge^{r,s}\tilde X}{\imm\delbar}$, there is a differential form $\alpha\in\wedge^{r,s-1}\tilde X$ such that $\psi=\delbar\alpha$. Hence, $\phi-\alpha\in\correnti^{r,s-1}\tilde X$ defines a class in $\frac{\ker\delbar \cap \correnti^{r,s-1}\tilde X}{\imm\delbar}\simeq \frac{\ker\delbar \cap \wedge^{r,s-1}\tilde X}{\imm\delbar}$, and hence $\phi-\alpha$ is a differential form up to a $\delbar$-exact form, and so $\phi$ is.

By conjugation, if $\psi\in\wedge^{r,s}\tilde X$ is a $\del$-closed differential form, then every solution $\phi\in\correnti^{r-1,s}$ of $\del\phi=\psi$ is a differential form up to $\del$-exact terms.

Now, let
$$ \omega^{p,q}\;=\;\de\eta\mod\imm\del\delbar\;\in\;\frac{\imm\de\cap\correnti^{p,q}X}{\imm\del\delbar} \;. $$
Decomposing $\eta=:\sum_{p,q}\eta^{p,q}$ in pure-type components, where $\eta^{p,q}\in\correnti^{p,q}\tilde X$, the previous equality is equivalent to the system
$$
\left\{
\begin{array}{cccccccc}
 && \del\eta^{p+q-1,0} &=&0 & \mod \imm\del\delbar && \\[5pt]
\delbar\eta^{p+q-\ell,\ell-1} &+& \del\eta^{p+q-\ell-1,\ell} &=& 0 &\mod\imm\del\delbar & \text{ for } & \ell\in\{1,\ldots,q-1\} \\[5pt]
\delbar\eta^{p,q-1} &+& \del\eta^{p-1,q} &=& \omega^{p,q} & \mod\imm\del\delbar && \\[5pt]
\delbar\eta^{\ell,p+q-\ell-1} &+& \del\eta^{\ell-1,p+q-\ell} &=& 0 &\mod\imm\del\delbar & \text{ for } & \ell\in\{1,\ldots,p-1\} \\[5pt]
\delbar\eta^{0,p+q-1} &&&=&0 &\mod\imm\del\delbar &&
\end{array}
\right. \;.
$$
By the above argument, we may suppose that, for $\ell\in\{0,\ldots, p-1\}$, the currents $\eta^{\ell,p+q-\ell-1}$ are differential form: indeed, they are differential form up to $\delbar$-exact terms, but $\delbar$-exact terms give no contribution in the system, which is modulo $\imm\del\delbar$. Analogously, we may suppose that, for $\ell\in\{0,\ldots,q-1\}$, the currents $\eta^{p+q-\ell-1,\ell}$ are differential form. Then we may suppose that $\omega^{p,q}=\delbar\eta^{p,q-1}+\del\eta^{p-1,q}$ is a differential form. Hence \eqref{eq:bc-orbifolds} is proven.

\smallskip

Now, we prove that, fixed a $G$-invariant Hermitian metric on $\tilde X$, the Bott-Chern cohomology of $\tilde X$ is isomorphic to the space of $\tilde\Delta_{BC}$-harmonic $G$-invariant forms on $X$. Indeed, since the elements of $G$ commute with $\del$, $\delbar$, $\del^*$, and $\delbar^*$, and hence with $\tilde\Delta_{BC}$, the following decomposition, \cite[Théorème 2.2]{schweitzer},
$$ \wedge^{\bullet,\bullet} X \;=\; \ker \tilde\Delta_{BC} \oplus \del\delbar\wedge^{\bullet-1,\bullet-1}X \oplus \left(\del^*\wedge^{\bullet+1,\bullet}X + \delbar^*\wedge^{\bullet,\bullet+1}X \right) $$
induces a decomposition
$$ \wedge^{\bullet,\bullet} \tilde X \;=\; \ker \tilde\Delta_{BC} \oplus \del\delbar\wedge^{\bullet-1,\bullet-1}\tilde X \oplus \left(\del^*\wedge^{\bullet+1,\bullet}\tilde X + \delbar^*\wedge^{\bullet,\bullet+1}\tilde X \right) \;; $$
more precisely, let $\alpha\in\wedge^{\bullet,\bullet}\tilde X$, that is, $\alpha$ is a $G$-invariant form on $X$; if $\alpha$ has a decomposition $\alpha=h_{\alpha}+\del\delbar\beta+\left(\del^*\gamma+\delbar^*\eta\right)$ with $h_\alpha,\beta,\gamma,\eta\in\wedge^{\bullet,\bullet}X$ such that $\tilde\Delta_{BC}h_\alpha=0$, then one has
\begin{eqnarray*}
\alpha \;=\; \frac{1}{\ord G}\sum_{g\in G}g^*\alpha &=& \left(\frac{1}{\ord G}\sum_{g\in G}g^*h_{\alpha}\right) +\del\delbar\left(\frac{1}{\ord G}\sum_{g\in G}g^*\beta\right) \\[5pt]
 && + \left(\del^*\left(\frac{1}{\ord G}\sum_{g\in G}g^*\gamma\right)+\delbar^*\left(\eta\frac{1}{\ord G}\sum_{g\in G}g^*\right)\right) \;,
\end{eqnarray*}
where $\frac{1}{\ord G}\sum_{g\in G}g^*h_{\alpha},\, \frac{1}{\ord G}\sum_{g\in G}g^*\beta,\, \frac{1}{\ord G}\sum_{g\in G}g^*\gamma,\, \eta\frac{1}{\ord G}\sum_{g\in G}g^*\in\wedge^{\bullet,\bullet}\tilde X$ and
$$ \tilde\Delta_{BC}\left(\frac{1}{\ord G}\sum_{g\in G}g^*h_{\alpha}\right) \;=\; \frac{1}{\ord G}\sum_{g\in G}g^*\left(\tilde\Delta_{BC}h_{\alpha}\right) \;=\; 0 \;.$$

As regards the Aeppli cohomology, one has the decomposition, \cite[\S2.c]{schweitzer},
$$ \wedge^{\bullet,\bullet}X \;=\; \ker\tilde\Delta_A \oplus \left(\del\wedge^{\bullet-1,\bullet}X + \delbar\wedge^{\bullet,\bullet-1} X\right) \oplus \left(\del\delbar\right)^*\wedge^{\bullet+1,\bullet+1} X \;,$$
and hence the decomposition
$$ \wedge^{\bullet,\bullet} \tilde X \;=\; \ker\tilde\Delta_A \oplus \left(\del\wedge^{\bullet-1,\bullet}\tilde X + \delbar\wedge^{\bullet,\bullet-1} \tilde X\right) \oplus \left(\del\delbar\right)^*\wedge^{\bullet+1,\bullet+1} \tilde X \;,$$
from which one gets the isomorphism $H^{\bullet,\bullet}_A\left(\tilde X\right)\simeq\ker\tilde\Delta_A$.

\smallskip

Finally, note that the Hodge-$*$-operator $*\colon \wedge^{\bullet_1,\bullet_2}\tilde X \to \wedge^{n-\bullet_2,n-\bullet_1}\tilde X$ sends $\tilde\Delta_{BC}$-harmonic forms to $\tilde\Delta_{A}$-harmonic forms, and hence it induces an isomorphism
$$ *\colon H^{\bullet_1,\bullet_2}_{BC}\left(\tilde X\right)\stackrel{\simeq}{\to} H^{n-\bullet_2,n-\bullet_1}_{A}\left(\tilde X\right) \;, $$
concluding the proof.
\end{proof}

\begin{rem}
 We note that another proof of the isomorphism
 $$ H^{p,q}_{BC}\left(\tilde X\right) \;\simeq\; \frac{\ker\left(\del \colon \correnti^{p,q}\tilde X \to \correnti^{p+1,q}\tilde X\right) \cap \ker \left(\delbar\colon \correnti^{p,q}\tilde X \to \correnti^{p,q+1}\tilde X\right)}{\imm\left(\del\delbar \colon \correnti^{p-1,q-1}\tilde X \to \correnti^{p,q}\tilde X\right)} \;, $$
 and a proof of the isomorphism
 $$ H^{p,q}_{A}\left(\tilde X\right) \;\simeq\; \frac{\ker\left(\del\delbar \colon \correnti^{p,q}\tilde X \to \correnti^{p+1,q+1}\tilde X\right)}{\imm\left(\del \colon \correnti^{p-1,q}\tilde X \to \correnti^{p,q}\tilde X\right) + \imm\left(\delbar \colon \correnti^{p,q-1}\tilde X \to \correnti^{p,q}\tilde X\right)} $$
 follow from the sheaf-theoretic interpretation of the Bott-Chern and Aeppli cohomologies, developed by J.-P. Demailly, \cite[\S V I.12.1]{demailly-agbook} and M. Schweitzer, \cite[\S4]{schweitzer}, see also \cite[\S3.2]{kooistra}.

 We recall that, for any $p,q\in\N$, the complex $\left(\mathcal{L}^\bullet_{\tilde X\, p,q},\, \de_{\mathcal{L}^\bullet_{\tilde X\, p,q}}\right)$ of sheaves is defined as
$$
\left(\mathcal{L}^\bullet_{\tilde X\, p,q},\, \de_{\mathcal{L}^\bullet_{\tilde X\, p,q}}\right) \;: \qquad
   \mathcal{A}^{0,0}_{\tilde X}
     \stackrel{\pr \circ \de}{\to}
   \bigoplus_{\substack{r+s=1 \\ r<p,\, s<q}} \mathcal{A}^{r,s}_{\tilde X}
     \to
   \cdots
     \stackrel{\pr \circ \de}{\to}
   \bigoplus_{\substack{r+s=p+q-2\\ r<p,\, s<q}} \mathcal{A}^{r,s}_{\tilde X}
     \stackrel{\del\delbar}{\to}
   \bigoplus_{\substack{r+s=p+q\\ r\geq p,\, s\geq q}} \mathcal{A}^{r,s}_{\tilde X}
     \stackrel{\de}{\to}
   \bigoplus_{\substack{r+s=p+q\\ r\geq p,\, s\geq q}} \mathcal{A}^{r,s}_{\tilde X}
     \to
   \cdots \;,
$$
and the complex $\left(\mathcal{M}^\bullet_{\tilde X\, p,q},\, \de_{\mathcal{M}^\bullet_{\tilde X\, p,q}}\right)$ of sheaves is defined as
$$
\left(\mathcal{M}^\bullet_{\tilde X\, p,q},\, \de_{\mathcal{M}^\bullet_{\tilde X\, p,q}}\right) \;: \qquad
   \mathcal{D}^{0,0}_{\tilde X}
     \stackrel{\pr \circ \de}{\to}
   \bigoplus_{\substack{r+s=1 \\ r<p,\, s<q}} \mathcal{D}^{r,s}_{\tilde X}
     \to
   \cdots
     \stackrel{\pr \circ \de}{\to}
   \bigoplus_{\substack{r+s=p+q-2\\ r<p,\, s<q}} \mathcal{D}^{r,s}_{\tilde X}
     \stackrel{\del\delbar}{\to}
   \bigoplus_{\substack{r+s=p+q\\ r\geq p,\, s\geq q}} \mathcal{D}^{r,s}_{\tilde X}
     \stackrel{\de}{\to}
   \bigoplus_{\substack{r+s=p+q\\ r\geq p,\, s\geq q}} \mathcal{D}^{r,s}_{\tilde X}
     \to
   \cdots \;,
$$
where $\pr$ denotes the projection onto the appropriate space.

Take $\phi$ a germ of a $\de$-closed $k$-form on $\tilde X$, with $k\in\N\setminus\{0\}$, that is, a germ of a $G$-invariant $k$-form on $X$; by the Poincaré lemma, see, e.g., \cite[I.1.22]{demailly-agbook}, there exists $\psi$ a germ of a $(k-1)$-form on $X$ such that $\phi=\de\psi$; since $\phi$ is $G$-invariant, one has
$$ \phi \;=\; \frac{1}{\ord G} \sum_{g\in G} g^*\phi \;=\; \frac{1}{\ord G} \sum_{g\in G} g^*\left(\de\psi\right) \;=\; \de\left(\frac{1}{\ord G} \sum_{g\in G} g^*\psi\right) \;, $$
that is, taking the germ of the $G$-invariant $(k-1)$-form
$$ \tilde\psi \;:=\; \frac{1}{\ord G} \sum_{g\in G} g^*\psi $$
on $X$, one gets a germ of a $(k-1)$-form on $\tilde X$ such that $\phi=\de\tilde\psi$. As regards the case $k=0$, one has straightforwardly that every ($G$-invariant) $\de$-closed function on $X$ is locally constant. The same argument applies for the sheaves of currents, by using the Poincaré lemma for currents, see, e.g., \cite[Theorem I.2.24]{demailly-agbook}.

Analogously, take $\phi$ a germ of a $\delbar$-closed $(p,q)$-form (respectively, bidimension-$(p,q)$-current) on $\tilde X$, with $q\in\N\setminus\{0\}$, that is, a germ of a $G$-invariant $(p,q)$-form (respectively, bidimension-$(p,q)$-current) on $X$; by the Dolbeault and Grothendieck lemma, see, e.g., \cite[I.3.29]{demailly-agbook}, there exists $\psi$ a germ of a $(p,q-1)$-form (respectively, bidimension-$(p,q-1)$-current) on $X$ such that $\phi=\delbar\psi$; since $\phi$ is $G$-invariant, one has
$$ \phi \;=\; \frac{1}{\ord G} \sum_{g\in G} g^*\phi \;=\; \frac{1}{\ord G} \sum_{g\in G} g^*\left(\delbar\psi\right) \;=\; \delbar\left(\frac{1}{\ord G} \sum_{g\in G} g^*\psi\right) \;, $$
that is, taking the germ of the $G$-invariant $(p,q-1)$-form (respectively, bidimension-$(p,q-1)$-current)
$$ \tilde\psi \;:=\; \frac{1}{\ord G} \sum_{g\in G} g^*\psi $$
on $X$, one gets a germ of a $(p,q-1)$-form (respectively, bidimension-$(p,q-1)$-current) on $\tilde X$ such that $\phi=\delbar\tilde\psi$. As regards the case $q=0$, one has that every ($G$-invariant) $\delbar$-closed bidimension-$(p,0)$-current on $X$ is locally a holomorphic $p$-form, see, e.g., \cite[I.3.29]{demailly-agbook}.

By the Poincaré lemma and the Dolbeault and Grothendieck lemma, one gets M. Schweitzer's lemma \cite[Lemme 4.1]{schweitzer}, which can be extended also to the context of orbifolds by using the same trick; this allows to prove that the map
$$ \left(\mathcal{L}^\bullet_{\tilde X\, p,q},\, \de_{\mathcal{L}^\bullet_{\tilde X\, p,q}}\right) \to \left(\mathcal{M}^\bullet_{\tilde X\, p,q},\, \de_{\mathcal{M}^\bullet_{\tilde X\, p,q}}\right) $$
of complexes of sheaves is a quasi-isomorphism, and hence, see, e.g., \cite[\S IV.12.6]{demailly-agbook}, for every $\ell\in\N$,
$$ \mathbb{H}^{\ell} \left(\tilde X;\,\left(\mathcal{L}^{\bullet}_{\tilde X\, p,q},\, \de_{\mathcal{L}^{\bullet}_{\tilde X\, p,q}}\right)\right) \;\simeq\; \mathbb{H}^{\ell} \left(\tilde X;\,\left(\mathcal{M}^{\bullet}_{\tilde X\, p,q},\, \de_{\mathcal{L}^{\bullet}_{\tilde X\, p,q}}\right)\right) \;.$$

Since, for every $k\in\N$, the sheaves $\mathcal{L}^k_{\tilde X\, p,q}$ and $\mathcal{M}^k_{\tilde X\, p,q}$ are fine (indeed, they are sheaves of $\left(\mathcal{C}^{\infty}_{\tilde X}\otimes_\R\C\right)$-modules over a paracompact space), one has, see, e.g., \cite[IV.4.19, (IV.12.9)]{demailly-agbook},
$$
\mathbb{H}^{p+q-1} \left(\tilde X;\,\left(\mathcal{L}^{\bullet}_{\tilde X\, p,q},\, \de_{\mathcal{L}^{\bullet}_{\tilde X\, p,q}}\right)\right) 
\;\simeq\;
\frac{\ker \left(\del\colon \wedge^{p,q}\tilde X \to \wedge^{p+1,q}\tilde X\right) \cap \ker \left(\delbar \colon \wedge^{p,q}\tilde X \to \wedge^{p,q+1}\tilde X \right)}{\imm\left(\del\delbar \colon \wedge^{p-1,q-1}\tilde X \to \wedge^{p,q}\tilde X\right)}
$$
and
$$
\mathbb{H}^{p+q-1} \left(\tilde X;\,\left(\mathcal{M}^{\bullet}_{\tilde X\, p,q},\, \de_{\mathcal{L}^{\bullet}_{\tilde X\, p,q}}\right)\right)
\;\simeq\;
\frac{\ker \left(\del\colon \correnti^{p,q}\tilde X \to \correnti^{p+1,q}\tilde X\right) \cap \ker \left(\delbar \colon \correnti^{p,q}\tilde X \to \correnti^{p,q+1}\tilde X \right)}{\imm\left(\del\delbar \colon \correnti^{p-1,q-1}\tilde X \to \correnti^{p,q}\tilde X\right)} \;,
$$
and
$$
\mathbb{H}^{p+q-2} \left(\tilde X;\,\left(\mathcal{L}^{\bullet}_{\tilde X\, p,q},\, \de_{\mathcal{L}^{\bullet}_{\tilde X\, p,q}}\right)\right)
\;\simeq\;
\frac{\ker \left(\del\delbar\colon \wedge^{p-1,q-1}\tilde X \to \wedge^{p,q}\tilde X\right)}{\imm\left(\del \colon \wedge^{p-2,q-1}\tilde X \to \wedge^{p-1,q-1}\tilde X\right) + \imm\left(\delbar \colon \wedge^{p-1,q-2}\tilde X \to \wedge^{p-1,q-1}\tilde X\right)} $$
and
$$
\mathbb{H}^{p+q-2} \left(\tilde X;\,\left(\mathcal{M}^{\bullet}_{\tilde X\, p,q},\, \de_{\mathcal{L}^{\bullet}_{\tilde X\, p,q}}\right)\right)
\;\simeq\;
\frac{\ker \left(\del\delbar\colon \correnti^{p-1,q-1}\tilde X \to \correnti^{p,q}\tilde X\right)}{\imm\left(\del \colon \correnti^{p-2,q-1}\tilde X \to \correnti^{p-1,q-1}\tilde X\right) + \imm\left(\delbar \colon \correnti^{p-1,q-2}\tilde X \to \correnti^{p-1,q-1}\tilde X\right)} \;,
$$
proving the stated isomorphisms.
\end{rem}

\section{Complex orbifolds satisfying the $\del\delbar$-Lemma}

We recall that a bounded double complex $\left(K^{\bullet,\bullet},\, \de',\, \de''\right)$ of vector spaces, whose associated simple complex is $\left(K^\bullet,\, \de\right)$ with $\de:=\de'+\de''$, is said \emph{to satisfy the $\de'\de''$-Lemma}, \cite{deligne-griffiths-morgan-sullivan}, if
$$ \ker\de'\cap\ker\de''\cap\imm\de=\imm\de'\de'' \;; $$
other equivalent conditions are provided in \cite[Lemma 5.15]{deligne-griffiths-morgan-sullivan}.

An orbifold $\tilde X$ is said \emph{to satisfy the $\del\delbar$-Lemma} if the double complex $\left(\wedge^{\bullet,\bullet}\tilde X,\, \del,\, \delbar\right)$ satisfies the $\del\delbar$-Lemma, that is, if every $\del$-closed $\delbar$-closed $\de$-exact form is $\del\delbar$-exact, namely, in other words, if the natural map $H^{\bullet,\bullet}_{BC}\left(\tilde X\right) \to H^\bullet_{dR}\left(\tilde X; \C\right)$ induced by the identity is injective.

Characterizations of compact complex manifolds satisfying the $\del\delbar$-Lemma in terms of their cohomological properties have been provided by P. Deligne, Ph. Griffiths, J. Morgan and D. Sullivan in \cite[Proposition 5.17, 5.21]{deligne-griffiths-morgan-sullivan}, and by the author and A. Tomassini in \cite[Theorem B]{angella-tomassini-3}. As a corollary of their characterization, P. Deligne, Ph. Griffiths, J. Morgan and D. Sullivan proved that, given $X$ and $Y$ compact complex manifolds of the same dimension and $f\colon X\to Y$ a holomorphic birational map, if $X$ satisfies the $\del\delbar$-Lemma, then also $Y$ satisfies the $\del\delbar$-Lemma, \cite[Theorem 5.22]{deligne-griffiths-morgan-sullivan}.

In this section, we extend \cite[Theorem 5.22]{deligne-griffiths-morgan-sullivan} to the case of orbifolds, by straightforwardly adapting a result by R.~O. Wells, \cite[Theorem 3.1]{wells}, to the orbifold case.

\begin{thm}[{see \cite[Theorem 3.1]{wells}}]\label{thm:wells-thm}
 Let $\tilde Y$ and $\tilde X$ be compact complex orbifolds of the same complex dimension, and let $\epsilon\colon \tilde Y\to \tilde X$ be a proper surjective morphism of complex orbifolds. Then the map $\epsilon\colon \tilde Y\to \tilde X$ induces injective maps
 $$
    \epsilon_{dR}^* \colon H^\bullet_{dR}\left(\tilde X;\R\right) \to H^\bullet_{dR}\left(\tilde Y;\R\right) \;,
     \qquad
    \epsilon_{\delbar}^*\colon H^{\bullet,\bullet}_{\delbar}\left(\tilde X\right) \to H^{\bullet,\bullet}_{\delbar}\left(\tilde Y\right) \;,
     \qquad \text{ and } \qquad
    \epsilon_{BC}^*\colon H^{\bullet,\bullet}_{BC}\left(\tilde X\right) \to H^{\bullet,\bullet}_{BC}\left(\tilde Y\right) \;.
 $$
\end{thm}

\begin{proof}
We follow closely the proof of \cite[Theorem 3.1]{wells} and adapt it to the orbifold case.

\paragrafo{1}{Notations}
The morphism $\epsilon \colon \tilde Y \to \tilde X$ of complex orbifolds induces morphisms
$$
   \epsilon^* \colon \wedge^{\bullet} \tilde X \to \wedge^{\bullet} \tilde Y
    \qquad \text{ and } \qquad
   \epsilon^* \colon \wedge^{\bullet,\bullet} \tilde X \to \wedge^{\bullet,\bullet} \tilde Y
$$
of $\R$-vector spaces and $\C$-vector spaces,
and hence, by duality,
$$
   \epsilon_* \colon \correnti_\bullet \tilde Y \to \correnti_\bullet \tilde X
    \qquad \text{ and } \qquad
   \epsilon_* \colon \correnti_{\bullet,\bullet} \tilde Y \to \mathcal{D}_{\bullet,\bullet}\tilde X \;.
$$
Moreover, recall that, for $X\in\left\{\tilde X, \, \tilde Y\right\}$, there are natural inclusions
$$
   T_\sspace \colon \wedge^\bullet X \to \correnti^{\bullet}X:=:\correnti_{2n-\bullet} X
    \qquad \text{ and } \qquad
   T_\sspace \colon \wedge^{\bullet,\bullet} X \to \correnti^{\bullet,\bullet}X:=:\correnti_{n-\bullet,n-\bullet} X \;,
$$
where $n$ is the complex dimension of $X$.

Both $\epsilon^*$ and $\epsilon_*$ commute with $\de$, $\del$ and $\delbar$, and hence they induce morphisms of complexes
$$
   \epsilon_{dR}^* \colon \left(\wedge^\bullet \tilde X,\, \de\right) \to \left(\wedge^\bullet \tilde Y ,\, \de\right)
    \qquad \text{ and } \qquad
   \epsilon^{dR}_* \colon \left(\correnti^\bullet \tilde Y,\, \de\right) \to \left(\correnti^\bullet \tilde X ,\, \de\right) \;,
$$
and, for any $p\in\N$,
$$
   \epsilon_{\delbar}^* \colon \left(\wedge^{p,\bullet} \tilde X,\, \delbar\right) \to \left(\wedge^{p,\bullet} \tilde Y ,\, \delbar\right)
    \qquad \text{ and } \qquad
   \epsilon^{\delbar}_* \colon \left(\correnti^{p,\bullet} \tilde Y,\, \delbar\right) \to \left(\correnti^{p,\bullet} \tilde X ,\, \delbar\right) \;,
$$
and, for any $p,q\in\N$,
$$
   \epsilon_{BC}^* \colon \left(\wedge^{p-1,q-1} \tilde X \stackrel{\del\delbar}{\to} \wedge^{p,q} \tilde X \stackrel{\del + \delbar}{\to} \wedge^{p+1,q} \tilde X \oplus \wedge^{p,q+1} \tilde X \right) \to \left(\wedge^{p-1,q-1} \tilde Y \stackrel{\del\delbar}{\to} \wedge^{p,q} \tilde Y \stackrel{\del + \delbar}{\to} \wedge^{p+1,q} \tilde Y \oplus \wedge^{p,q+1} \tilde Y \right)
$$
and
$$
   \epsilon^{BC}_* \colon \left(\correnti^{p-1,q-1} \tilde Y \stackrel{\del\delbar}{\to} \correnti^{p,q} \tilde Y \stackrel{\del + \delbar}{\to} \correnti^{p+1,q} \tilde Y \oplus \wedge^{p,q+1} \tilde Y \right) \to \left(\correnti^{p-1,q-1} \tilde X \stackrel{\del\delbar}{\to} \correnti^{p,q} \tilde X \stackrel{\del + \delbar}{\to} \correnti^{p+1,q} \tilde X \oplus \correnti^{p,q+1} \tilde X \right) \;;
$$
hence, they induce morphisms between the corresponding cohomologies:
$$
 \epsilon_{dR}^* \colon H^\bullet_{dR}\left(\tilde X;\R\right) \to H^\bullet_{dR}\left(\tilde Y;\R\right) \;,
  \quad
 \epsilon_{\delbar}^*\colon H^{\bullet,\bullet}_{\delbar}\left(\tilde X\right) \to H^{\bullet,\bullet}_{\delbar}\left(\tilde Y\right) \;,
  \quad \text{ and } \quad
 \epsilon_{BC}^*\colon H^{\bullet,\bullet}_{BC}\left(\tilde X\right) \to H^{\bullet,\bullet}_{BC}\left(\tilde Y\right) \;.
$$

Recall that $T_\sspace$ commutes with $\de$, $\del$ and $\delbar$, and hence it induces, for $X\in\left\{\tilde X ,\, \tilde Y\right\}$, morphisms
$$ T_\sspace \colon \left(\wedge^\bullet X,\, \de\right) \to \left(\correnti^\bullet X,\, \de\right) \;,$$
and, for any $p\in\N$,
$$
   T_\sspace \colon \left(\wedge^{p,\bullet} X,\, \delbar\right) \to \left(\correnti^{p,\bullet} X ,\, \delbar\right) \;,
$$
and, for any $p,q\in\N$,
$$
   T_\sspace \colon \left(\wedge^{p-1,q-1} X \stackrel{\del\delbar}{\to} \wedge^{p,q} X \stackrel{\del + \delbar}{\to} \wedge^{p+1,q} X \oplus \wedge^{p,q+1} X \right) \to \left(\correnti^{p-1,q-1} X \stackrel{\del\delbar}{\to} \correnti^{p,q} X \stackrel{\del + \delbar}{\to} \correnti^{p+1,q} X \oplus \wedge^{p,q+1} X \right) \;;
$$
by \cite[Theorem 1]{satake}, by \cite[page 807]{baily-2}, and by Theorem \ref{thm:bc}, these maps are in fact quasi-isomorphisms.

\paragrafo{3}{It holds $\epsilon_*\, T_\sspace\, \epsilon^*=\mu \cdot T_\sspace$ for some $\mu\in\N\setminus\{0\}$}
Indeed, consider the diagrams
$$
\xymatrix{
 \wedge^\bullet \tilde Y  \ar[r]^{T_\sspace}  &   \correnti^\bullet \tilde Y \ar[d]^{\epsilon_*}  \\
 \wedge^\bullet \tilde X  \ar[r]_{T_\sspace} \ar[u]^{\epsilon^*}  &   \correnti^\bullet \tilde X 
} \;,
\qquad \text{ respectively } \qquad
\xymatrix{
 \wedge^{\bullet,\bullet} \tilde Y  \ar[r]^{T_\sspace}  &   \correnti^{\bullet,\bullet} \tilde Y \ar[d]^{\epsilon_*}  \\
 \wedge^{\bullet,\bullet} \tilde X  \ar[r]_{T_\sspace} \ar[u]^{\epsilon^*}  &   \correnti^{\bullet,\bullet} \tilde X 
} \;.
$$
One has that there exists a proper analytic subset $S_{\tilde Y}$ of $\tilde Y\setminus \Sing\left(\tilde Y\right)$ such that $S_{\tilde Y}$ has measure zero in $\tilde Y$ and
$$ \epsilon\lfloor_{\tilde Y \setminus \left(\Sing\left(\tilde Y\right) \cup S_{\tilde Y} \right)} \colon \tilde Y \setminus \left(\Sing\left(\tilde Y\right) \cup S_{\tilde Y} \right) \to \tilde X \setminus \left(\Sing\left(\tilde X\right) \cup \epsilon\left(S_{\tilde Y}\right)\right) $$
is a finitely-sheeted covering mapping of sheeting number $\mu\in\N\setminus\{0\}$.
Let $\mathcal{U}:=:\left\{U_\alpha\right\}_{\alpha\in\mathcal{A}}$ be an open covering of $\tilde X\setminus \left(\Sing\left(\tilde X\right) \cup \epsilon\left(S_{\tilde Y}\right)\right)$, and let $\left\{\rho_\alpha\right\}_{\alpha\in\mathcal{A}}$ be an associated partition of unity.
For every $\varphi, \psi\in\wedge^{\bullet,\bullet} \tilde X$, one has that
\begin{eqnarray*}
 \left\langle \epsilon_*\, T_\sspace\, \epsilon^* \varphi ,\, \psi \right\rangle &=& \left\langle T_\sspace \, \epsilon^* \varphi ,\, \epsilon^* \psi \right\rangle
 \;=\; \int_{\tilde Y} \epsilon^* \varphi \wedge \epsilon^* \psi
 \;=\; \int_{\tilde Y} \epsilon^* \left(\varphi\wedge\psi\right)
 \;=\; \int_{\tilde Y - \left(\Sing\left(\tilde Y\right) \cup S_{\tilde Y}\right) } \epsilon^* \left(\varphi\wedge\psi\right) \\[5pt]
 &=& \sum_{\alpha\in\mathcal{A}} \int_{\pi^{-1}\left(U_\alpha\right)} \epsilon^* \left(\rho_\alpha\left(\varphi\wedge\psi\right)\right) 
 \;=\; \sum_{\alpha\in\mathcal{A}} \sum_{\sharp\left\{U \in \mathcal{U} \st \pi^{-1}\left(U\right)=\pi^{-1}\left(U_\alpha\right)\right\}} \int_{U_\alpha} \rho_\alpha\left(\varphi\wedge\psi\right) \\[5pt]
 &=& \mu \cdot \int_{\tilde X - \left(\Sing\left(\tilde X\right) \cup \epsilon\left(S_{\tilde Y}\right)\right)} \varphi\wedge\psi
 \;=\; \mu \cdot \int_{\tilde X} \varphi\wedge\psi
 \;=\; \left\langle \mu \, T_\sspace \varphi ,\, \psi \right\rangle \;,
\end{eqnarray*}
and hence one gets that
$$ \epsilon_*\, T_\sspace\, \epsilon^* \;=\; \mu \cdot T_\sspace \;.$$

\paragrafo{4}{Conclusion}
Hence one has the diagrams
$$
\xymatrix{
 \frac{\ker \left(\de \colon \wedge^\bullet \tilde X \to \wedge^{\bullet+1}\tilde X\right)}{\imm \left(\de \colon \wedge^{\bullet-1} \tilde X \to \wedge^{\bullet}\tilde X\right)} \ar[r]^{T_\sspace}_{\simeq}  & \frac{\ker \left(\de \colon \correnti^\bullet \tilde X \to \correnti^{\bullet+1}\tilde X\right)}{\imm \left(\de \colon \correnti^\bullet \tilde X \to \correnti^{\bullet+1}\tilde X\right)} \ar[d]^{\epsilon^{dR}_*}  \\
 \frac{\ker \left(\de \colon \wedge^\bullet \tilde Y \to \wedge^{\bullet+1} \tilde Y\right)}{\imm \left(\de \colon \wedge^{\bullet-1} \tilde Y \to \wedge^{\bullet} \tilde Y\right)} \ar[r]_{T_\sspace}^{\simeq} \ar[u]^{\epsilon_{dR}^*}  &  \frac{\ker \left(\de \colon \correnti^\bullet \tilde Y \to \correnti^{\bullet+1} \tilde Y\right)}{\imm \left(\de \colon \correnti^\bullet \tilde Y \to \correnti^{\bullet+1} \tilde Y\right)}
} \;,
$$
such that
$$ \epsilon^{dR}_* \, T_\sspace \, \epsilon_{dR}^* \;=\; \mu \cdot T_\sspace \;,$$
and
$$
\xymatrix{
 \frac{\ker \left(\delbar \colon \wedge^{\bullet,\bullet} \tilde X \to \wedge^{\bullet,\bullet+1}\tilde X\right)}{\imm \left(\delbar \colon \wedge^{\bullet,\bullet-1} \tilde X \to \wedge^{\bullet,\bullet} \tilde X\right)} \ar[r]^{T_\sspace}_{\simeq}  & \frac{\ker \left(\delbar \colon \correnti^{\bullet,\bullet} \tilde X \to \correnti^{\bullet,\bullet+1}\tilde X\right)}{\imm \left(\de \colon \correnti^{\bullet,\bullet-1} \tilde X \to \correnti^{\bullet,\bullet}\tilde X\right)} \ar[d]^{\epsilon^{\delbar}_*}  \\
 \frac{\ker \left(\delbar \colon \wedge^{\bullet,\bullet} \tilde Y \to \wedge^{\bullet,\bullet+1} \tilde Y\right)}{\imm \left(\delbar \colon \wedge^{\bullet,\bullet-1} \tilde Y \to \wedge^{\bullet,\bullet} \tilde Y\right)} \ar[r]_{T_\sspace}^{\simeq} \ar[u]^{\epsilon_{\delbar}^*}  &  \frac{\ker \left(\delbar \colon \correnti^{\bullet,\bullet} \tilde Y \to \correnti^{\bullet,\bullet+1} \tilde Y\right)}{\imm \left(\delbar \colon \correnti^{\bullet,\bullet-1} \tilde Y \to \correnti^{\bullet,\bullet} \tilde Y\right)}
} \;,
$$
such that
$$ \epsilon^{\delbar}_* \, T_\sspace \, \epsilon_{\delbar}^* \;=\; \mu \cdot T_\sspace \;,$$
and
$$
\xymatrix{
 \frac{\ker \left(\del\delbar \colon \wedge^{\bullet,\bullet} \tilde X \to \wedge^{\bullet+1,\bullet+1}\tilde X\right)}{\imm \left(\del \colon \wedge^{\bullet-1,\bullet} \tilde X \to \wedge^{\bullet,\bullet} \tilde X\right) + \imm \left(\delbar \colon \wedge^{\bullet,\bullet-1} \tilde X \to \wedge^{\bullet,\bullet} \tilde X\right)} \ar[r]^{T_\sspace}_{\simeq}  & \frac{\ker \left(\del\delbar \colon \correnti^{\bullet,\bullet} \tilde X \to \correnti^{\bullet+1,\bullet+1}\tilde X\right)}{\imm \left(\del \colon \correnti^{\bullet-1,\bullet} \tilde X \to \correnti^{\bullet,\bullet} \tilde X\right) + \imm \left(\de \colon \correnti^{\bullet,\bullet-1} \tilde X \to \correnti^{\bullet,\bullet}\tilde X\right)} \ar[d]^{\epsilon^{BC}_*}  \\
 \frac{\ker \left(\del\delbar \colon \wedge^{\bullet,\bullet} \tilde Y \to \wedge^{\bullet+1,\bullet+1} \tilde Y\right)}{\imm \left(\del \colon \wedge^{\bullet-1,\bullet} \tilde Y \to \wedge^{\bullet,\bullet} \tilde Y\right) + \imm \left(\delbar \colon \wedge^{\bullet,\bullet-1} \tilde Y \to \wedge^{\bullet,\bullet} \tilde Y\right)} \ar[r]_{T_\sspace}^{\simeq} \ar[u]^{\epsilon_{BC}^*}  &  \frac{\ker \left(\del\delbar \colon \correnti^{\bullet,\bullet} \tilde Y \to \correnti^{\bullet+1,\bullet+1} \tilde Y\right)}{\imm \left(\del \colon \correnti^{\bullet-1,\bullet} \tilde Y \to \correnti^{\bullet,\bullet} \tilde Y\right) + \imm \left(\delbar \colon \correnti^{\bullet,\bullet-1} \tilde Y \to \correnti^{\bullet,\bullet} \tilde Y\right)}
} \;,
$$
such that
$$ \epsilon^{BC}_* \, T_\sspace \, \epsilon_{BC}^* \;=\; \mu \cdot T_\sspace \;. $$
Since $T_\sspace$ are isomorphisms in cohomology, one gets that
$$
    \epsilon_{dR}^* \colon H^\bullet_{dR}\left(\tilde X;\R\right) \to H^\bullet_{dR}\left(\tilde Y;\R\right) \;,
     \qquad
    \epsilon_{\delbar}^*\colon H^{\bullet,\bullet}_{\delbar}\left(\tilde X\right) \to H^{\bullet,\bullet}_{\delbar}\left(\tilde Y\right) \;,
     \qquad \text{ and } \qquad
    \epsilon_{BC}^*\colon H^{\bullet,\bullet}_{BC}\left(\tilde X\right) \to H^{\bullet,\bullet}_{BC}\left(\tilde Y\right) \;.
$$
are injective.
\end{proof}

\begin{cor}\label{cor:deldelbar-lemma}
 Let $\tilde Y$ and $\tilde X$ be compact complex orbifolds of the same dimension, and let $\epsilon\colon \tilde Y\to \tilde X$ be a proper surjective morphism of complex orbifolds. If $\tilde Y$ satisfies the $\del\delbar$-Lemma, then also $\tilde X$ satisfies the $\del\delbar$-Lemma.
\end{cor}

\begin{proof}
One has the commutative diagram
$$
\xymatrix{
 H^{\bullet,\bullet}_{BC}\left(\tilde X\right) \ar[r]^{\epsilon^*_{BC}}_{\text{1:1}} \ar[d]^{\id^*_{\tilde X}} & H^{\bullet,\bullet}_{BC}\left(\tilde Y\right) \ar[d]^{\id^*_{\tilde Y}}_{\text{1:1}} \\
 H^{\bullet}_{dR}\left(\tilde X;\C\right) \ar[r]^{\epsilon^*_{dR}}_{\text{1:1}} & H^{\bullet}_{dR}\left(\tilde Y;\C\right)
}
$$
where $\id^*_{\tilde X} \colon H^{\bullet,\bullet}_{BC}\left(\tilde X\right) \to H^{\bullet}_{dR}\left(\tilde X;\C\right)$ and $\id^*_{\tilde Y} \colon H^{\bullet,\bullet}_{BC}\left(\tilde Y\right) \to H^{\bullet}_{dR}\left(\tilde Y;\C\right)$ are the natural maps induced in cohomology by the identity. Since $\id^*_{\tilde Y} \colon H^{\bullet,\bullet}_{BC}\left(\tilde Y\right) \to H^{\bullet}_{dR}\left(\tilde Y;\C\right)$ is injective by the assumption that $\tilde Y$ satisfies the $\del\delbar$-Lemma, and $\epsilon^*_{BC} \colon H^{\bullet,\bullet}_{BC}\left(\tilde X\right) \to H^{\bullet,\bullet}_{BC}\left(\tilde Y\right)$ and $\epsilon^*_{dR} \colon H^{\bullet}_{dR}\left(\tilde X;\C\right) \to H^{\bullet}_{dR}\left(\tilde Y;\C\right)$ are injective by Theorem \ref{thm:wells-thm}, we get that also $\id^*_{\tilde X} \colon H^{\bullet,\bullet}_{BC}\left(\tilde X\right) \to H^{\bullet}_{dR}\left(\tilde X;\C\right)$ is injective, and hence $\tilde X$ satisfies the $\del\delbar$-Lemma.
\end{proof}


\begin{thebibliography}{10}

\bibitem{angella-1}
D. Angella, The cohomologies of the Iwasawa manifold and of its small deformations, \textsc{doi}: \texttt{10.1007/s12220-011-9291-z}, to appear in {\em J. Geom. Anal.}.

\bibitem{angella-tomassini-3}
D. Angella, A. Tomassini, On the $\partial\overline\partial$-Lemma and Bott-Chern cohomology, \textsc{doi}: \texttt{10.1007/s00222-012-0406-3}, to appear in {\em Invent. Math.}.

\bibitem{babenko-taimanov}
I.~K. Babenko, I.~A. Ta{\u\i}manov, On nonformal simply connected symplectic manifolds, {\em Sibirsk. Mat. Zh.}, \textbf{41} (2000), no.~2, 253--269, translation in {\em Siberian Math. J.} \textbf{41} (2000), no.~2, 204--217.

\bibitem{baily}
W.~L. Baily, The decomposition theorem for $V$-manifolds, {\em Amer. J. Math.} {\bfseries 78} (1956), no.~4, 862--888.

\bibitem{baily-2}
W.~L. Baily, On the quotient of an analytic manifold by a group of analytic homeomorphisms, {\em Proc. Nat. Acad. Sci. U. S. A.} {\bfseries 40} (1954), no.~9, 804--808.

\bibitem{bochner}
S. Bochner, Compact groups of differentiable transformations, {\em Ann. of Math. (2)} \textbf{46} (1945), no.~3, 372--381.

\bibitem{cavalcanti-fernandez-munoz}
G.~R. Cavalcanti, M. Fern{\'a}ndez, V. Mu{\~n}oz, Symplectic resolutions, {L}efschetz property and formality, {\em Adv. Math.}, \textbf{218} (2008), no.~2, 576--599.

\bibitem{deligne-griffiths-morgan-sullivan}
P. Deligne, Ph. Griffiths, J. Morgan, D.  Sullivan, Real homotopy theory of K\"ahler manifolds, {\em Invent. Math.} \textbf{29} (1975), no.~3, 245--274.

\bibitem{demailly-agbook}
J.-P. Demailly, \emph{{C}omplex {A}nalytic and {D}ifferential {G}eometry}, \url{http://www-fourier.ujf-grenoble.fr/~demailly/manuscripts/agbook.pdf}, 2009.

\bibitem{fernandez-munoz}
M. Fern{\'a}ndez, V. Mu{\~n}oz, An {$8$}-dimensional nonformal, simply connected, symplectic manifold, {\em Ann. of Math. (2)}, \textbf{167} (2008), no.~3, 1045--1054.

\bibitem{frolicher}
A. Fr\"olicher, Relations between the cohomology groups of Dolbeault and topological invariants, {\em Proc. Nat. Acad. Sci. U.S.A.} \textbf{41} (1955), no.~9, 641--644.

\bibitem{joyce-invent}
D.~D. Joyce, Compact {$8$}-manifolds with holonomy {${\rm Spin}(7)$}, {\em Invent. Math.} \textbf{123} (1996), no.~3, 507--552.

\bibitem{joyce-jdg-1-2-g2}
D.~D. Joyce, Compact {R}iemannian {$7$}-manifolds with holonomy {$G_2$}. {I}, {II}, {\em J. Differ. Geom.} \textbf{43} (1996), no.~2, 291--328, 329--375.

\bibitem{joyce-jdg-spin7}
D.~D. Joyce, A new construction of compact 8-manifolds with holonomy {${\rm Spin}(7)$}, {\em J. Differ. Geom.} \textbf{53} (1999), no.~1, 89--130.

\bibitem{joyce-ext}
D.~D. Joyce, \emph{Compact Manifolds with Special Holonomy}, Oxford Mathematical Monographs, Oxford University Press, Oxford, 2000.

\bibitem{joyce-red}
D.~D. Joyce,
{\em Riemannian holonomy groups and calibrated geometry},
Oxford Graduate Texts in Mathematics \textbf{12}, Oxford University Press, Oxford, 2007.

\bibitem{kodaira-spencer-3}
K. Kodaira, D.~C. Spencer, On deformations of complex analytic structures. III. Stability theorems for complex structures, {\em Ann. of Math. (2)} \textbf{71} (1960), no.~1, 43--76.

\bibitem{kooistra}
R. Kooistra, Regulator currents on compact complex manifolds, Thesis (Ph.D.)–University of Alberta (Canada), 2011.

\bibitem{mccleary}
J. McCleary, {\em A user's guide to spectral sequences}, Second edition, Cambridge Studies in Advanced Mathematics, \textbf{58}, Cambridge University Press, Cambridge, 2001.

\bibitem{raissy-master-thesis}
J. Raissy, {\em Normalizzazione di campi vettoriali olomorfi}, Tesi di Laurea Specialistica, Università di Pisa,
\url{http://etd.adm.unipi.it/theses/available/etd-06022006-141206/}, 2006.

\bibitem{satake}
I. Satake, On a generalization of the notion of manifold, {\em Proc. Nat. Acad. Sci. U.S.A.} {\bfseries 42} (1956), no.~6, 359--363.

\bibitem{schweitzer}
M. Schweitzer, Autour de la cohomologie de Bott-Chern, \texttt{arXiv:0709.3528v1 [math.AG]}.

\bibitem{wells}
R.~O. Wells, Comparison of de Rham and Dolbeault cohomology for proper surjective mappings, {\em Pacific J. Math.} {\bfseries 53} (1974), no.~1, 281--300.
\end{thebibliography}
\end{document}